\documentclass[a4paper]{amsart}

\usepackage{amsmath,amssymb,amsthm,a4wide}

\usepackage{tikz}
\usetikzlibrary{calc,intersections,through,backgrounds}

\usepackage{hyperref,caption}
\usepackage{graphicx}
\usepackage{graphics}
\usepackage{epsfig}

\newtheorem{theorem}{Theorem}
\newtheorem*{thm}{Theorem}
\newtheorem{lemma}{Lemma}

\begin{document}

\title[Stability Estimates for Truncated Fourier/Laplace Transform]{Stability Estimates for Truncated \\Fourier and Laplace Transforms}
\author[R.R. Lederman]{Roy R. Lederman}
\address{Roy R. Lederman, Program in Applied and Computational Mathematics, Princeton University, Fine Hall, Washington Road, Princeton, NJ 08544, USA} \email{roy@math.princeton.edu}

\author[S. Steinerberger]{Stefan Steinerberger}
\address{Stefan Steinerberger, Department of Mathematics, Yale University, 10 Hillhouse Avenue, New Haven, CT 06511, USA} \email{stefan.steinerberger@yale.edu}

\begin{abstract} We prove sharp stability estimates for the Truncated Laplace Transform and Truncated Fourier Transform. The argument combines an approach recently introduced by Alaifari, Pierce and the second author for the truncated Hilbert transform with classical results of Bertero, Gr\"unbaum, Landau, Pollak and Slepian. In particular, we prove there is a universal constant
$c >0$ such that for all $f \in L^2(\mathbb{R})$ with compact support in $[-1,1]$ normalized to $\|f\|_{L^2[-1,1]} = 1$
$$ \int_{-1}^{1}{|\widehat{f}(\xi)|^2d\xi} \gtrsim \left(c\left\|f_x \right\|_{L^2[-1,1]} \right)^{- c\left\|f_x \right\|_{L^2[-1,1]}}$$
The inequality is sharp in the sense that there is an infinite sequence of orthonormal counterexamples if $c$ is chosen too small. The question whether and to which extent similar inequalities
hold for generic families of integral operators remains open.
\end{abstract}
\maketitle
\vspace{-15pt}
\section{Introduction} 
\subsection{Introduction.} Given a compact operator $T:H \rightarrow H$ on a Hilbert space $H$, compactness implies that the inversion problem, i.e. reconstructing $x$ from $y$ in
$$ Tx = y$$
is ill-posed: small changes in $y$ may lead to arbitrarily large changes in $x$. The simplest example is perhaps that of integral operators on $L^2(\mathbb{R})$ where integration acts as a smoothing
process and makes inversion of the operator difficult. Of particular importance is the Hilbert transform
$$ (Hf)(x) = \frac{1}{\pi}\mbox{p.v.}\int_{\mathbb{R}}{\frac{f(y)}{x-y}dy},$$
which satisfies $\| Hf\|_{L^2(\mathbb{R})} = \|f\|_{L^2(\mathbb{R})}$. However, in practice, measurements
have to be taken from a compact interval and this motivates the definition of the truncated Hilbert transform: using $\chi_I$ to
denote the characteristic function on an interval $I \subset \mathbb{R}$, the truncated Hilbert transform $H_T:L^2(I) \rightarrow L^2(J)$ on the
intervals $I,J \subset \mathbb{R}$ is given by
$$ H_T = \chi_{J} H(f \chi_{I}).$$
Whenever the intervals $I$ and $J$ are disjoint, the singularity of the kernel never comes into play and the operator is highly smoothing: indeed, if $I$ and $J$
are disjoint, the operator becomes \textit{severely} ill-posed and the singular values decay exponentially fast. The inversion problem is ill-behaved even on finite-dimensional subspaces: \textit{every}  subspace $V \subset L^2(I)$ contains some $0 \neq f \in V$ with
$$ \| H_T f\|_{L^2(J)} \leq c_1 e^{-c_2 \dim(V)} \| f \|_{L^2(I)} \qquad \mbox{for some}~c_1, c_2 > 0~\mbox{depending only on}~I,J.$$

\begin{figure}[h!]
\begin{center}
\begin{tikzpicture}[xscale=9,yscale=1.1]
\draw [ultra thick, domain=0:1, samples = 300] plot (\x, {-0.15269*sin(2*pi*\x r)    +    0.4830*sin(3*pi*\x r)  +  0.3084*sin(4*pi*\x r)  + 0.80509*sin(5*pi*\x r)}  );
\draw [thick, domain=0:1] plot (\x, {0}  );
\filldraw (0,0) ellipse (0.006cm and 0.048cm);
\node at (0,-0.3) {0};
\filldraw (1,0) ellipse (0.006cm and 0.048cm);
\node at (1,-0.3) {1};
\end{tikzpicture}
\caption{A function $f$ on $[0,1]$ with $\|Hf\|^2_{L^2([2,3])} \sim 10^{-7}\|f\|^2_{L^2([0,1])}$} 
\end{center}
\end{figure}
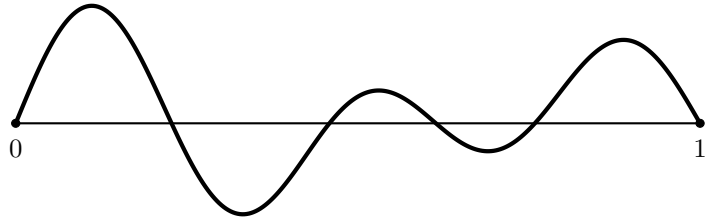

This strong form of ill-posedness makes it very easy to construct bad examples: take any finite
orthonormal set $\left\{\phi_1, \phi_2, \dots, \phi_n \right\} \subset L^2(I)$, By linearity, we have for any scalar $a_1, \dots, a_n$ that
$$ \left\| H_T \left(\sum_{k=1}^{n}{a_k \phi_k}\right)\right\|_{L^2(J)}^2 = \sum_{i, j = 1}^{n}{a_i a_j \left\langle H_T \phi_i, H_T \phi_j \right\rangle_{L^2(J)}}$$
which is a simple quadratic form. Finding the eigenvector corresponding to the smallest eigenvalue of the Gramian $G = (\left\langle H_T \phi_i, H_T \phi_j \right\rangle)_{i,j=1}^{n}$
produces a suitable linear combination of  $\left\{\phi_1, \phi_2, \dots, \phi_n \right\}$ for which $\|H_Tf\|_{L^2(J)} \ll \|f\|_{L^2(I)}$. The strong degree of ill-posedness guarantees that the smallest eigenvalue decays
exponentially in $n$ independently of the orthonormal basis. Recently, Alaifari, Pierce and the second author \cite{al} showed that it is nonetheless possible to guarantee some control by proving a new type of stability
estimate for the  Hilbert transform: for disjoint intervals $I,J \subset \mathbb{R}$
$$  \|H f\|_{L^2(J)} \geq c_1 \exp{\left(-c_2\frac{ \|f_x\|_{L^2(I)}}{\|f\|_{L^2(I)}}\right)} \| f \|_{L^2(I)},$$
where the constants $c_1, c_2$ depend only on the intervals $I,J$.
This estimate guarantees that the only way for $Hf$ to be substantially smaller than $f$ is the presence of oscillations. If one reconstructs data $f$ from
measurements $g$ (the equation being $H_T f = g$), then a small error $f + h$ yields
$$ H_T(f +h) = H_Tf + H_T h= g + H_T h.$$
The stability estimate implies that one can guarantee to distinguish $f$ from $f+h$ when $h$ has few oscillations.
The only existing result in this direction is \cite{al} for the Hilbert transform.

\section{Main results}

The purpose of our paper is to combine the argument developed by Alaifari, Pierce and the second author \cite{al} with classical results of Bertero \& Gr\"unbaum \cite{gru1}, Landau \& Pollak \cite{pr2, pr3}  and Slepian \& Pollak \cite{pr1}  to establish such stability estimate in three other cases: we give essentially sharp stability estimates for the Truncated Laplace Transform, the Adjoint Truncated Laplace Transform and the Truncated Fourier Transform. While this shows that this class of stability estimates exist in a wider context, the question of whether such results could be 'generically' true (i.e. for a wide class of integral operators) remains open.

\subsection{Truncated Laplace Transform} 
 The truncated Laplace transform $\mathcal{L}_{a,b}:L^2[a,b] \rightarrow L^2[0,\infty]$ is defined via
$$
 (\mathcal{L}_{a,b}f)(s) = \int_{a}^{b}{e^{-s t} f(t) dt},
$$
where $0 < a < b < \infty$. 
The operator $\mathcal{L}_{a,b}$ is compact and its image is dense in $L^2[0,\infty]$. We show
that if $\|\mathcal{L}_{a,b} f\|_{L^2[0, \infty]} \ll \|f\|_{L^2[a,b]}$,
then this is due to the presence of oscillations.

\begin{figure}[h!]
\begin{center}
\begin{tikzpicture}[xscale=9,yscale=1.1]
\draw [ultra thick, domain=1:2, samples = 300] plot (\x, {-0.0707*sin(pi*\x r)    - 0.421*sin(2*pi*\x r)  +  0.2137*sin(3*pi*\x r)  + 0.8783*sin(4*pi*\x r)}  );
\draw [thick, domain=1:2] plot (\x, {0}  );
\filldraw (1,0) ellipse (0.006cm and 0.048cm);
\node at (1,-0.3) {1};
\filldraw (2,0) ellipse (0.006cm and 0.048cm);
\node at (2,-0.3) {2};
\end{tikzpicture}
\caption{A function $f$ on $[1,2]$ with $\| \mathcal{L}_{1,2} f \|^2_{L^2[0,\infty]} \sim 10^{-8}\|f\|^2_{L^2([1,2])}$.} 
\end{center}
\end{figure}
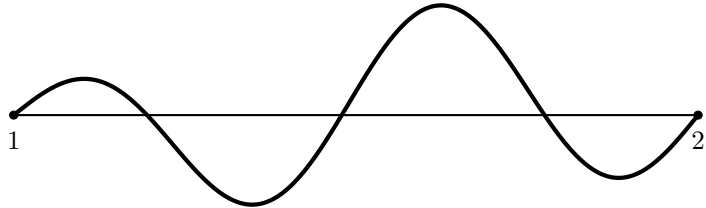

\begin{theorem} There exist $c_1, c_2>0$, depending only on $a,b$, so that for all real-valued $f \in H^1[a,b]$
$$ \| \mathcal{L}_{a,b} f \|_{L^2[0,\infty]} \geq c_1 \exp{\left(-c_2\frac{ \|f_x\|_{L^2[a,b]}}{\|f\|_{L^2[a,b]}}\right)}\|f\|_{L^2[a,b]}.$$
\end{theorem}
The result is sharp up to constants: if $c_2$ is chosen sufficiently small, then for every $c_1 > 0$ there is an infinite orthonormal sequence of functions for which the inequality fails. The proof proceeds similarly as in \cite{al} with a crucial ingredient for Laplace transforms coming from a  a 1985 paper of Bertero \& Gr\"unbaum \cite{gru1}.

\subsection{Adjoint Truncated Laplace Transform.} The adjoint operator $\mathcal{L}_{a,b}^*:L^2[0,\infty] \rightarrow L^2[a,b]$ 
$$ (\mathcal{L}_{a,b}^*f)(s) = \int_{0}^{\infty}{e^{-s t} f(t) dt}.$$
is very different in structure. We seek a lower bound on $\|\mathcal{L}_{a,b}^*f\|_{L^2[a,b]}$ in terms of $\| f \|_{L^2[0, \infty]}$: if $f$ is supported far away from the
origin, then the exponentially decaying kernel will induce rapid decay even if no oscillations are present (additional oscillations can, of course, further decrease the size of  $\|\mathcal{L}_{a,b}^*f\|_{L^2[a,b]}$).
Any lower bound will therefore have to incorporate where the function is localized and the natural framework for this are weighted estimates.

\begin{theorem} There exist $c_1, c_2$, depending only on $a,b$, so that for all real-valued $f \in H^2[0, \infty]$
$$ \| \mathcal{L}_{a,b}^* f \|_{L^2[a,b]} \geq c_1 \exp{\left(-c_2\frac{ \|x f_{xx}\|_{L^2[0,\infty]}  + \|x f_{x}\|_{L^2[0,\infty]} + \|x f_{}\|_{L^2[0,\infty]} + \| f_{}\|_{L^2[0,\infty]}   }{\|f\|_{L^2[0,\infty]}}\right)}\|f\|_{L^2[0,\infty]}.$$
\end{theorem}
The result is again sharp in the sense that there are counterexamples for every $c_1 > 0$ if the constant $c_2$ is smaller than some fixed positive
constant depending on $a,b$.

\subsection{Truncated Fourier Transform} Let $\mathcal{F}_T: L^2[-1,1] \rightarrow L^2[-1,1]$ be given by
$$ \mathcal{F}_T = \chi_{[-1,1]}\mathcal{F}\left(\chi_{[-1,1]} f\right)$$
where, as usual, $\mathcal{F}$ denotes the Fourier transform
$$ (\mathcal{F} f)(\xi) = \int_{\mathbb{R}}^{}{f(x) e^{i  \xi x}dx}.$$
The Fourier transform of a compactly supported function is analytic and cannot vanish on an open set. Since it does not vanish on any open set, this yields
$$ \int_{-1}^{1}{|\widehat{f}(\xi)|^2d\xi} > 0$$
for every nonzero $f \in L^2[-1,1]$. The expression can certainly be small because $\widehat{f}$ can have all its $L^2-$mass far away from the origin: however, if $\widehat{f}$ has its $L^2-$mass
far away from the origin, $f$ oscillates on $[-1,1]$. We give a quantitative description of this phenomenon.

\begin{theorem} There exist $c_1, c_2 > 0$ such that for all real-valued $f \in H^1[-1,1]$
$$ \int_{-1}^{1}{|\widehat{f}(\xi)|^2d\xi}\geq c_1\left(c_2 \frac{ \left\| f_x \right\|_{L^2[-1,1]}}{\|f\|_{L^2[-1,1]}} \right)^{-c_2\frac{\left\| f_x \right\|_{L^2[-1,1]}}{\|f\|_{L^2[-1,1]}} } \int_{-1}^{1}{|f(x)|^2dx}.$$
\end{theorem}

\begin{figure}[h!]
\begin{center}
\begin{tikzpicture}[xscale=4.5,yscale=1.1]
\draw [ultra thick, domain=-1:1, samples = 300] plot (\x, {0.00055*cos(pi*\x r)    +  0.0824*cos(2*pi*\x r)  +  0.6196*cos(3*pi*\x r)  + 0.7805*cos(4*pi*\x r)}  );
\draw [thick, domain=-1.05:1.05] plot (\x, {0}  );
\filldraw (-1,0) ellipse (0.012cm and 0.048cm);
\node at (-1,-0.3) {-1};
\filldraw (0,0) ellipse (0.012cm and 0.048cm);
\node at (0,-0.3) {0};
\filldraw (1,0) ellipse (0.012cm and 0.048cm);
\node at (1,-0.3) {1};
\end{tikzpicture}
\caption{A function $f$ on $[-1,1]$ with $\| \mathcal{F}_T f \|^2_{L^2[-1,1]} \sim 10^{-18}\|f\|^2_{L^2([-1,1])}$.} 
\end{center}
\end{figure}
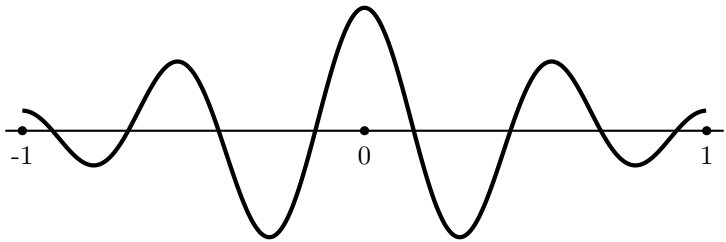

We are not aware of any such results in the literature,
however, the result is certainly close in spirit to the question to which degree simultaneous localization in space and frequency is possible. An example is Nazarov's
quantitative form \cite{naz} of the Amrein-Berthier theorem \cite{am} (see also \cite{bene}): for any $S, \Sigma \subset \mathbb{R}$ with
finite measure and any $f \in L^2(\mathbb{R})$ it is not possible for $f$ to be too strongly localized in $S$ and $\widehat{f}$ to be too
strongly localized in $\Sigma$
$$ \left\|f \chi_{\mathbb{R} \setminus S} \right\|^2_{L^2(\mathbb{R})} + \left\|\widehat{f} \chi_{\mathbb{R} \setminus \Sigma} \right\|^2_{L^2(\mathbb{R})} \geq \frac{e^{-133 |S| |\Sigma|}}{133} \| f\|^2_{L^2(\mathbb{R})}.$$
The proof of Theorem 3 makes use of \textit{prolate spheroidal wave functions} introduced by Landau, Pollak and Slepian \cite{pr2, pr3, pr1, pr4, pr5}. They appear naturally in the
Landau-Pollak uncertainty principle \cite{pr3} which states that if
$\mbox{supp}(\widehat{f}) \subset [-1,1]$
and
$$ \int_{|x| \geq T}{|f(x)|^2 dx} \leq \varepsilon \|f\|_{L^2(\mathbb{R})},~\mbox{then}
\qquad \|f - \pi(f) \|_{L^2} \leq 49\varepsilon^2 \|f\|_{L^2},$$
where $\pi$ is the projection onto a $(4\left\lfloor T \right\rfloor +1)-$dimensional subspace spanned by the first elements of a particular \textit{universal} orthonormal basis $(\phi_n)_{n \in \mathbb{N}}$ (these are the prolate spheroidal wave functions).  \\

\textbf{Outline of the paper.} \S 3 gives a high-level overview of the argument and provides two easy inequalities for real functions that will be used in the proofs.  \S 4 explains the underlying machinery specially required to prove Theorem 1 and gives the full proof. A very similar argument allows to prove Theorem 2 and we describe the necessary modifications in \S 5. \S 6 gives a proof of Theorem 3. $c_1, \dots, c_5$  are positive constants, $\sim$ denotes equivalence up to constants.

%%%%%%%%%%%%%%%%%%%%%%%%%%%%%%%%%%%%%%%%%%%%%%
%
%%%%%%%%%%%%%%%%%%%%%%%%%%%%%%%%%%%%%%%%%%%%%%

\section{Outline of the arguments}

\subsection{The overarching structure.}
The proofs (also for the result in \cite{al}) have the same underlying structure: we use a $T^* T$ argument and the fact that
there is a differential operator $D$ whose eigenfunctions coincide with the eigenfunctions of $T^* T$. This allows us to exploit the structure of
the differential operator to analyze the decomposition of a generic function into the orthonormal basis of singular functions.
More precisely: we are interested in establishing lower bounds for an 
injective operator between two Hilbert spaces $T:H_1 \rightarrow H_2$. In all these cases, we assume that
\begin{enumerate}
\item we control the decay of the eigenvalues of $T^*T$ from below,
\item there is a differential operator $D:H_1 \rightarrow H_1$ with the same eigenfunctions as $T^*T$
\item and we can control the growth of eigenvalues $\lambda_n$ of $D$.
\end{enumerate}
Let us denote the $L^2-$normalized eigenfunctions of $D$ (which are also eigenfunctions of $T^*T$) by $(u_n)_{n=1}^{\infty}$. They form an orthonormal basis
of $L^2$ in all situations that are of interest to us. Furthermore, we will use the spectral theorem
$$ \left\langle D f, f\right\rangle = \sum_{n=1}^{\infty}{\lambda_n \left| \left\langle f, u_n \right\rangle \right|^2}$$
and explicit information on the growth of the eigenvalues $\lambda_n$. We can furthermore, using integration by parts and the structure of $D$, control the action of $D$ in the Sobolev space $H^{s}$
$$  \left\langle D f, f\right\rangle \sim    \| f\|^2_{H^s}.$$
The useful insight is that this implies that the eigenfunction $(u_n)_{n=1}^{\infty}$ explore the phase space in a way that is analogous to classical eigenfunctions of the Laplacian: low-energy eigenfunctions
have small derivatives. In particular, if $Df$ is small, then at least some of the projections $|\left\langle f, u_n \right\rangle|$ have to be big for $n$ somewhat small. Conversely, functions whose $L^2-$energy
is mostly concentrated on high-frequency eigenfunctions $(u_n)_{n \geq N}$ have $ |\left\langle D f, f\right\rangle|$ large. The next Lemma makes this precise.

\begin{lemma}[Low oscillation implies low frequency] If $\lambda_n \geq c_1  n^{2}$ and $|\left\langle D f, f\right\rangle| \leq c_2 \|f_x\|_{L^2}^2$ for some $0 < c_1, c_2 < \infty$, then there exists a constant $0 < c < \infty$ such that
$$ \sum_{n \leq c \frac{ \|f_x\|_{L^2}}{\|f\|_{L^2}}  }^{}{ \left| \left\langle f, u_n \right\rangle\right|^2} \geq \frac{ \| f\|^2_{L^2}}{2}.$$
\end{lemma}
\begin{proof} Both inequalities have the same scaling under the multiplication with scalars $f \rightarrow \lambda f$, so we can assume w.l.o.g. that $\|f\|_{L^2} = 1$. Trivially,
\begin{align*}
  \sum_{n \geq c_3\|f_x\|_{L^2}}^{}{ \lambda_n  \left| \left\langle f, u_n \right\rangle\right|^2} &\geq   \sum_{n \geq c_3 \|f_x\|_{L^2}}^{}{ c_1 n^{2}  \left| \left\langle f, u_n \right\rangle\right|^2} \\
&\geq c_1 \left(c_3 \|f_x\|_{L_2}\right)^2   \sum_{n \geq c_3 \|f_x\|_{L^2}}^{}{  \left| \left\langle f, u_n \right\rangle\right|^2} 
\end{align*}
However, we also clearly have that
$$ \sum_{n \geq c_3 \|f_x\|_{L^2}}^{}{\lambda_n  \left| \left\langle f, u_n \right\rangle\right|^2}  \leq   \sum_{n=1}^{\infty}{\lambda_n  \left| \left\langle f, u_n \right\rangle\right|^2}= |\left\langle D f, f\right\rangle| \leq   c_2\| f_x\|^2_{L^2}.$$
As a consequence
$$   \sum_{n \geq c_3 \|f_x\|_{L^2}}^{}{ \left(c_3 \|f_x\|_{L^2}\right)^{2}  \left| \left\langle f, u_n \right\rangle\right|^2}  \leq \frac{c_2}{c_1 c_3^2},$$
which can be made smaller than $1/2$ for a suitable choice of $c_3$ (depending on $c_1,c_2$). Since the $(u_n)_{n=1}^{\infty}$ form an orthonormal system
$$ 1 = \|f\|_{L^2}^2 =  \sum_{n=1 }^{\infty}{\left| \left\langle f, u_n \right\rangle\right|^2},~\mbox{we get} \quad \sum_{n \leq \sqrt{\frac{2c_2}{c_1}} \frac{ \|f_x\|_{L^2}}{   \|f\|_{L^2}  }  }^{}{ \left| \left\langle f, u_n \right\rangle\right|^2} \geq \frac{ \| f\|^2_{L^2}}{2}.$$
\end{proof}

We may not know the eigenfunctions $(u_n)_{n=1}^{\infty}$ but we can ensure that for any function $f$ half
of its $L^2-$mass of the expansion will be contained in the subspace
$$ \mbox{span}\left\{u_n: n \leq c \frac{ \|f_x\|_{L^2}}{\|f\|_{L^2}} \right\} \subset H_1.$$
The second step of the argument invokes decay of the eigenvalues $\mu_n$ of $T^* T$ via
\begin{align*}
\| T f\|^2_{H_2} &= \left\langle Tf, Tf \right\rangle_{H_2} = \left\langle T^*Tf, f \right\rangle_{H_1} = \sum_{n =1 }^{\infty}{ \mu_n |\left\langle f, u_n \right\rangle|^2}
\end{align*}
and combining this with the previous argument to obtain
$$ \sum_{n =1 }^{\infty}{ \mu_n |\left\langle f, u_n \right\rangle|^2} \geq \sum_{n \leq c \frac{ \|f_x\|_{L^2}}{\|f\|_{L^2}}   }^{\infty}{ \mu_n |\left\langle f, u_n \right\rangle|^2} \geq 
\mu_{ c \frac{ \|f_x\|_{L^2}}{\|f\|_{L^2}}  } \sum_{n \leq  c \frac{ \|f_x\|_{L^2}}{\|f\|_{L^2}}  }^{\infty}{|\left\langle f, u_n \right\rangle |^2} \geq  \frac{\mu_{ c \frac{ \|f_x\|_{L^2}}{\|f\|_{L^2} }} \|f\|^2_{L^2(H_1)} }{2}.  
$$
\textit{Sharpness of results.} It is not difficult to see that these types of arguments are actually sharp (up to constant) if $f=u_n$. This will immediately imply sharpness of our results: if constants
in the statement are chosen too small, then the inequality will fail for $(u_n)_{n \geq N}$ for some $N$ sufficiently large. While this is not our main focus, there is quite
a bit of additional research on precise asymptotics of the constants and how they depend on the intervals (see \cite{led0}).

\subsection{An easy inequality.} All our proofs will have a natural case-distinction: either the function changes sign on the interval $[a,b]$ or it does not. If it changes sign, then
we can use standard arguments to bound all arising terms by $\|f_x\|_{L^2[a,b]}$ which simplifes the expressions. 

\begin{lemma} Let $[a,b] \subset \mathbb{R}$. If $f:[a,b]$ is differentiable and changes sign on $[a,b]$, then
$$ \|f\|_{L^{\infty}[a,b]} \leq  \sqrt{b-a} \|f_x\|_{L^2[a,b]}.$$
\end{lemma}
\begin{proof} Let us assume $f(x_0) = 0$ for some $x_0 \in [a,b]$. Then, for every $x \in [a,b]$, using Cauchy-Schwarz
$$ |f(x)| = \left| \int_{x_0}^{x}{f'(z) dz} \right| \leq \int_{x_0}^{x}{|f'(z)| dz} \leq \sqrt{b-a} \|f_x\|_{L^2[a,b]}.$$
\end{proof}

If $f$ does \textit{not} change sign, then we cannot bound low-regularity terms like $\|f\|_{L^2}$ by high-regularity terms like $\|f_x\|_{L^2[a,b]}$. However, there is also no cancellation in
the integral operator and arguments specifically taylored to the integral operators will admit easy lower bounds in terms of the $L^1-$norm. The next inequality shows that the lower bounds we obtain in the Theorems are much smaller than the $L^1-$norm so that we may treat both cases at the same time.

\begin{lemma} Let $[a,b] \subset \mathbb{R}$. Then, for every $c_2 > 0$, there exists a $c_1 > 0$ (depending on $c_2, a, b$) such that for all nonnegative, differentiable $f:[a,b] \rightarrow \mathbb{R}_{+}$ 
$$  \int_{a}^{b}{ f(x) dx}  \geq  c_1 \exp{\left(-c_2\frac{ \|f_x\|_{L^2[a,b]}}{\|f\|_{L^2[a,b]}}\right)}\|f\|_{L^2[a,b]}.$$
\end{lemma}
\begin{proof} Squaring both sides of the desired inequality and using
$$ \|f\|_{L^2[a,b]}^2 = \int_{a}^{b}{f(x)^2 dx} \leq \|f\|_{L^{\infty}} \int_{a}^{b}{f(x) dx}$$
shows that the desired statement is implied by the stronger inequality
$$ \|f\|_{L^{\infty}[a,b]} \leq \frac{1}{c_1^2} \exp{\left(c_2\frac{ \|f_x\|_{L^2[a,b]}}{\|f\|_{L^2[a,b]}}\right)} \int_{a}^{b}{ f(x) dx}.$$
The inequality is invariant under multiplication with scalars $f \rightarrow c f$, which allows us to assume w.l.o.g. that $\|f\|_{L^{\infty}[a,b]} = 1$.
Let us now take $J \subset [a,b]$ to be the largest possible interval such that $f$ assumes the value 1 on the boundary of $J$ and
the value $1/2$ on the other boundary point. If no such interval exists, then the original inequality trivially holds with $c_1 = \sqrt{b-a}/2$ since
$$  \int_{a}^{b}{ f(x) dx} \geq \frac{b-a}{2}  \geq \frac{ \sqrt{b-a}}{2} \|f\|_{L^2[a,b]} \geq \frac{ \sqrt{b-a}}{2} \exp{\left(-c_2\frac{ \|f_x\|_{L^2[a,b]}}{\|f\|_{L^2[a,b]}}\right)} \|f\|_{L^2[a,b]}.$$
Suppose now that $J$ exists. Clearly, 
$$ \int_{a}^{b}{f(x)dx} \geq  \int_{J}^{}{f(x)dx} \geq \frac{|J|}{2} \qquad \mbox{and} \qquad \|f\|_{L^2[a,b]} \leq \sqrt{b-a}.$$
It remains to bound $\|f_x\|_{L^2[a,b]}$ from below. We use the trivial estimate $\|f_x\|_{L^2[a,b]} \geq \|f_x\|_{L^2(J)}$ and argue that
among all functions on the interval $J$ assuming the values 1 and $1/2$ on the boundary, the linear function yields the smallest value for $\|f_x\|_{L^2(J)}$.
The existence of a minimizing function is obvious because of compactness. The minimizer $g$ has to satisfy the Euler-Lagrange equation, which simplifies to $g_{xx} = 0$.
This implies 
$$\|f_x\|_{L^2(J)} \geq \left\| \left(1 - \frac{x}{2|J|}\right)_{x}\right\|_{L^2[0, |J|]} =  \frac{1}{2\sqrt{|J|}}.$$
Altogether, we have
$$  \frac{1}{c_1^2} \exp{\left(c_2\frac{ \|f_x\|_{L^2[a,b]}}{\|f\|_{L^2[a,b]}}\right)} \int_{a}^{b}{ f(x) dx} \geq  \frac{1}{c_1^2} \exp{\left(\frac{c_2}{2 \sqrt{|J|} \sqrt{b-a} }\right)} \frac{|J|}{2}.$$
However, for every choice of $a,b,c_2>0$ such that $a<b$, the function $h:\mathbb{R}_{+} \rightarrow \mathbb{R}_{+}$ given by
$$ h(x) =  \exp{\left(\frac{c_2}{2 \sqrt{x} \sqrt{b-a} }\right)} \frac{x}{2}$$
is uniformly bounded away from 0 with a lower bound that only depends on $a,b,c_2$. This allows us to pick $c_1$ in such a way that 
$$  \frac{1}{c_1^2} \exp{\left(\frac{c_2}{2 \sqrt{|J|} \sqrt{b-a} }\right)} \frac{|J|}{2} \geq 1 = \|f\|_{L^{\infty}[a,b]}$$
independently of the length of the interval $J$. This gives the result.
\end{proof}

\section{Proof of Theorem 1}

\subsection{The Differential Operator.}

Let $f \in L^2[a,b]$ be arbitrary. We write
$$
 \| \mathcal{L}_{a,b}f \|_{L^2[0,\infty]}^2 = \left\langle  \mathcal{L}_{a,b}f,  \mathcal{L}_{a,b}f \right\rangle_{L^2[0,\infty]} =  \left\langle   \mathcal{L}_{a,b}^*  \mathcal{L}_{a,b}f, f \right\rangle_{L^2[0,\infty]}.$$
A brief computation shows that
\begin{align*}
 \| \mathcal{L}_{a,b}f \|_{L^2[0,\infty]}^2 = \int_{0}^{\infty}{  \left(\int_{a}^{b}{ e^{-s t} f(t) dt} \right)^2 } &= \int_{a}^{b}{ \int_{a}^{b}{    f(t)f(r) \left( \int_{0}^{\infty}{e^{-s t} e^{s r} ds} \right) dt dr} }\\
&=  \int_{a}^{b}{ \left( \int_{a}^{b}{ \frac{f(r)}{ r + t} dr} \right) f(t) dt}
 \end{align*} 
and thus
$$ (\mathcal{L}_{a,b}^*  \mathcal{L}_{a,b} f)(t) = \int_{a}^{b}{\frac{f(s)}{t+s} ds}.$$ 

\subsection{The differential operator.} The next ingredient, due to Bertero \& Gr\"unbaum \cite{gru1}, is crucial: they discovered that $\mathcal{L}_{a,b}^*  \mathcal{L}_{a,b}$ commutes with the differential operator 
$$
D_t =  \frac{d}{dt} \left(  (t^2-a^2)(b^2-t^2) \frac{d}{dt} \right) - 2(t^2-a^2).
$$
\begin{lemma}[Bertero \& Gr\"unbaum \cite{gru1}]\label{lem:gru_commute1}
For $f \in C^2[a,b]$
$$
\mathcal{L}_{a,b}^*  \mathcal{L}_{a,b} D_t f = D_t \mathcal{L}_{a,b}^*  \mathcal{L}_{a,b} f
$$
\end{lemma}
\begin{proof} The proof is an explicit computation starting with
$$ A = \int_a^b \frac{1}{t+s} \frac{d}{ds} \left(  (s^2-a^2)(b^2-s^2) \frac{d}{ds} \right) f(s) ds.$$
Integration by parts yields that
\begin{align*}
 A &= \frac{1}{t+s}(s^2-a^2)(b^2-s^2) \frac{d}{ds} f(s) \big|_a^b - \int_a^b  \left( \frac{d}{ds} \frac{1}{t+s} \right)  \left(  (s^2-a^2)(b^2-s^2) \frac{d}{ds} f(s) \right) ds  \\
&=  - \int_a^b  \left(  (s^2-a^2)(b^2-s^2) \frac{d}{ds} \frac{1}{t+s} \right)  \left(  \frac{d}{ds} f(s) \right)  ds \end{align*}
A second integration by parts gives that
$$ A = 
\int_a^b  \left( \frac{d}{ds} (s^2-a^2)(b^2-s^2) \frac{d}{ds} \frac{1}{t+s} \right)    f(s) ds. $$
Summarizing, we have just shown that
$$ \int_a^b \frac{1}{t+s} \frac{d}{ds} \left(  (s^2-a^2)(b^2-s^2) \frac{d}{ds} \right) f(s) ds = \int_a^b  \left( \frac{d}{ds} (s^2-a^2)(b^2-s^2) \frac{d}{ds} \frac{1}{t+s} \right) f(s) ds,$$
which we can also write as
$$
\int_a^b  \left( \frac{1}{t+s} \right) D_s f(s) ds = \int_a^b  \left( D_s \frac{1}{t+s} \right) f(s) ds.
$$

A simple computation shows that 
$$
 D_s \frac{1}{t+s} =  D_t \frac{1}{t+s}
$$
and thus, by linearity,
$$ D_t \int_a^b  \frac{1}{t+s}  f(s) ds = \int_a^b \frac{1}{t+s}  D_s f(s) ds.$$
\end{proof}

Using this commutation property in combination with some additional considerations regarding multiplicity (to ensure there are no degeneracies),
one can obtain the following result.
\begin{thm}[Bertero \& Gr\"unbaum, \cite{gru1}]
The eigenfunctions  $(u_n)_{n=1}^{\infty}$ of $D_t$ coincide with the eigenfunctions of $\mathcal{L}_{a,b}^*  \mathcal{L}_{a,b}$.
\end{thm}
This means we can now restrict ourselves to an analysis of the differential operator
$$
D_t =  -\frac{d}{dt} \left(  (t^2-a^2)(b^2-t^2) \frac{d}{dt} \right) + 2(t^2-a^2) \qquad \mbox{on}~[a,b],
$$
where we switched the sign to make it positive-definite.
We note that the basic bound $\lambda_n \sim n^2$ on the eigenvalues that follows immediately from standard spectral theory.

\begin{lemma}[Standard estimate, cf. \cite{led0}]   There exists $ c_1 > 0$ depending on $a,b$ such that the eigenvalues of $D_t$ on $[a,b]$
satisfy
$$ \lambda_n \geq c_3 n^2 .$$
\end{lemma}

\subsection{Proof of Theorem 1.} 
\begin{proof}
The proof combines the various ingredients. We assume w.l.o.g. that $\| f\|_{L^2[a,b]} = 1$. Integration by part gives, for differentiable $f$,
\begin{align*}  \left\langle D f, f \right\rangle  &\leq  \int_{a}^{b}{ (t^2-a^2)(b^2-t^2) \left( \frac{d}{dt} f(t)\right)^2 + 2(t^2-a^2) f(t)^2 dt } \\
&\leq (b^2-a^2)^2\| f_x\|^2_{L^2[a,b]} + 2(b^2-a^2) \| f\|^2_{L^2[a,b]}.\end{align*}
We distinguish two cases: (1) $f$ has a root in $[a,b]$ or (2) $f$ has no roots in $[a,b]$.
We start with the first case. Then Lemma 2 implies
$$ \| f\|^2_{L^2[a,b]} \leq (b-a) \| f\|^2_{L^{\infty}[a,b]} \leq (b-a)^2 \| f_x\|^2_{L^2[a,b]}$$
and thus
\begin{align*} \left| \left\langle D f, f \right\rangle \right| &\leq (b^2-a^2)^2\| f_x\|^2_{L^2[a,b]} + 2(b^2-a^2) \| f\|^2_{L^2[a,b]} \\
&\leq \left( (b^2-a^2)^2 + 2(b^2-a^2)(b-a)^2 \right)\|f_x\|_{L^2[a,b]}^2.\end{align*}
At the same time, since the eigenfunctions form a basis, we may also write
\begin{align*} \left\langle D f, f \right\rangle
= \sum_{n=1}^{\infty}{\lambda_n |\left\langle f, v_n \right\rangle|^2}
  \end{align*}
Altogether, we have, using the lower bound $\lambda_n \geq c_3 n^2$ that
$$   \sum_{n=1}^{\infty}{c_3  n^2 |\left\langle f, v_n \right\rangle|^2} \leq    \sum_{n=1}^{\infty}{\lambda_n |\left\langle f, v_n \right\rangle|^2}  = |\left\langle Df, f\right\rangle| \leq c_4 \| f_x\|^2_{L^2[a,b]} .$$
As a consequence, we can use Lemma 1 to deduce that the Littlewood-Paley projection onto low frequencies contains a positive fraction of the $L^2-$mass
$$   \sum_{n \leq c_5 \|f_x\|_{L^2[a,b]}}^{}{ |\left\langle f, v_n \right\rangle|^2} \geq \frac{1}{2}\|f\|_{L^2[a,b]}^2.$$ 
The argument can now be concluded as follows: it is known that the eigenvalues of $\mathcal{L}_{a,b}^*  \mathcal{L}_{a,b}$ decay exponentially (for estimates, see \cite{led0,led3})
 and we have also just established that a positive proportion of the $L^2-$mass lies at suitably small frequencies. We write 
\begin{align*}
\| \mathcal{L}_{a,b}  f\|^2_{L^2[0, \infty]} &= \left\langle \mathcal{L}_{a,b} f, \mathcal{L}_{a,b} f \right\rangle_{L^2[0, \infty]} = \left\langle \mathcal{L}_{a,b}^* \mathcal{L}_{a,b} f, f \right\rangle_{L^2[a,b]} = \sum_{n =1 }^{\infty}{ \mu_n |\left\langle f, u_n \right\rangle|^2},
\end{align*}
where $(\mu_n)_{n=1}^{\infty}$ are the eigenvalues of $ \mathcal{L}_{a,b}^* \mathcal{L}_{a,b}: L^2[a,b] \rightarrow L^2[a,b]$ and $(u_n)_{n=1}^{\infty}$ is the associated sequence of eigenfunctions. We bound
\begin{align*}
 \sum_{n =1 }^{\infty}{ \mu_n |\left\langle f, u_n \right\rangle|^2} &\geq \sum_{n \leq c_5 \|f_x\|_{L^2[a,b]}  }^{\infty}{ \mu_n |\left\langle f, u_n \right\rangle|^2} \\
&\geq \mu_{ c_5 \|f_x\|_{L^2[a,b]}}  \sum_{n \leq  c_5 \|f_x\|_{L^2[a,b]}}^{\infty}{|\left\langle f, u_n \right\rangle |^2} \\
&\geq  \frac{\mu_{ c_5 \|f_x\|_{L^2[a,b]}}}{2}.
\end{align*}
It is well-known (see e.g. \cite{led0}) that the singular values decay exponentially
$$ \mu_n \geq c_1 e^{-c_2 n},$$
where the constants $c_1, c_2$ only depend on the interval. This yields 
$$ \| \mathcal{L}_{a,b} f \|_{L^2[0,\infty]}^2 = \left\langle   \mathcal{L}_{a,b}^*  \mathcal{L}_{a,b}f, f \right\rangle \geq c_1 \exp{\left(-c_2 \|f_x\|_{L^2[a,b]}\right)}\|f\|^2_{L^2[a,b]}$$
for functions satisfying $\|f\|_{L^2[a,b]} = 1$ which, in turn, implies that for general $f \in L^2[a,b]$ 
$$ \| \mathcal{L}_{a,b} f \|_{L^2[0,\infty]}^2  \geq c_1 \exp{\left(-c_2\frac{ \|f_x\|_{L^2[a,b]}}{\|f\|_{L^2[a,b]}}\right)}\|f\|^2_{L^2[a,b]}.$$
It remains to consider the second case. In that case, $f$ cannot change sign. We assume w.l.o.g. that it is always positive and bound
\begin{align*}
\| \mathcal{L}_{a,b}f \|_{L^2[0,\infty]}^2 =  \int_{a}^{b}{ \left( \int_{a}^{b}{ \frac{f(r)}{ r + t} dr} \right) f(t) dt} \geq  \int_{a}^{b}{ \left( \int_{a}^{b}{ \frac{f(r)}{ b+b} dr} \right) f(t) dt}
= \frac{1}{2b} \left( \int_{a}^{b}{ f(t) dt} \right)^2.
\end{align*}
However, here Lemma 3 immediately yields that for every $c_2 > 0$ and all $a<b$ there exists a $c_1$ (depending only on $a,b,c_2$) such that for all differentiable $f:[a,b] \rightarrow \mathbb{R}$ that
do not change sign
$$  \left(\int_{a}^{b}{ f(t) dt} \right)^2 \geq  c_1 \exp{\left(-c_2\frac{ \|f_x\|_{L^2[a,b]}}{\|f\|_{L^2[a,b]}}\right)}\|f\|^2_{L^2[a,b]}.$$
\end{proof}
It is not difficult to check that all the steps are sharp up to constants for $f=u_n$ and this guarantees the sharpness of our Theorem up to constants.

\section{Proof of Theorem 2}

\subsection{The differential operator.} The self-adjoint operator $\mathcal{L}_{a,b}  \mathcal{L}_{a,b}^*$ has the form
$$  (\mathcal{L}_{a,b}  \mathcal{L}_{a,b}^*f)(s) = \int_{0}^{\infty}{\frac{e^{-a(s+t)} - e^{-b(s+t)}}{s + t} f(t) dt}$$
and its eigenfunctions $(u_n)_{n=1}^{\infty}$ now correspond to a weighted fourth-order differential operator. 
\begin{lemma}[Bertero \& Gr\"unbaum, \cite{gru1}]
The eigenfunctions $(u_n)$ form a basis in $L^2[0, \infty]$. Moreover, if 
$$ \hat{D}_t= -\frac{d^2}{dt^2}\left(t^2 \frac{d^2}{dt^2}  \right) + (a^2 + b^2)\frac{d}{dt} \left(t^2  \frac{d}{dt} \right) + (-a^2b^2 t^2 + 2a^2),$$
then the eigenfunctions of $D_t$ are also given by $(u_n)_{n=1}^{\infty}$.
\end{lemma}
The argument is very similar to before, the crucial ingredient is the commutation relation
$$ \hat{D}_t   \mathcal{L}_{a,b}  \mathcal{L}_{a,b}^*   =    \mathcal{L}_{a,b}  \mathcal{L}_{a,b}^*    \hat{D}_t   \  $$
which reduces, after several integration by parts, to
$$ \hat{D}_s \frac{e^{-a(s+t)} - e^{-b(s+t)}}{s + t} =  \hat{D}_t \frac{e^{-a(s+t)} - e^{-b(s+t)}}{s + t}.$$

\subsection{Proof of Theorem 2}

\begin{proof} The overall structure mirrors that of Theorem 2, it suffices to record the differences. 
We proceed as before and normalize to $\|f\|_{L^2} = 1$. Using integration by parts, we can bound
$$ \left\langle \hat{D}_t f, f\right\rangle = \left\langle  -\frac{d^2}{dt^2}\left(t^2 \frac{d^2}{dt^2} f(t) \right) + (a^2 + b^2)\frac{d}{dt} \left(t^2  \frac{d}{dt} f(t) \right) + (-a^2b^2 t^2 + 2a^2)f, f \right\rangle$$
from above by
\begin{align*} J &\leq \int_{0}^{\infty}{ t^2 \left(\frac{d^2}{dt^2} f\right)^2 dt} +  (a^2+b^2) \int_{0}^{\infty}{ t^2 \left(\frac{d}{dt} f(t)\right)^2 dt} +  (a^2+b^2) \int_{0}^{\infty}{( a^2 b^2 t^2+2a^2) f(t)^2 dt}\\
&\leq \| x f_{xx} \|^2_{L^2[0,\infty]} + c_1  \|x f_x\|^2_{L^2[0,\infty]} + c_2\|x f\|^2_{L^2[0,\infty]} + c_3 \|f\|^2_{L^2[0,\infty]}
\end{align*}
The direct spectral analysis of the operator $\hat{D}_t$ seems trickier, however, we can use
$$\hat{D}_t(f)= \mathcal{L}_{a,b} \circ D_t \circ (\mathcal{L}_{a,b})^{-1}$$
to conclude that $\hat{D}_t$ has the same eigenvalues as $D_t$, where $D_t$ is the differential operator from the proof of Theorem 1. In particular, $\lambda_n \geq c_4 n^2$ for some $c_4 > 0$.
Therefore 
\begin{align*}
  \sum_{n=1}^{\infty}{c_4  n^2 |\left\langle f, v_n \right\rangle|^2} \leq  \sum_{n=1}^{\infty}{\lambda_n |\left\langle f, v_n \right\rangle|^2}  \leq &\| x f_{xx} \|^2_{L^2[0,\infty]} + c_1  \|x f_x\|^2_{L^2[0,\infty]} \\
&+ c_2\|x f\|^2_{L^2[0,\infty]} + c_3 \|f\|^2_{L^2[0,\infty]}.\end{align*}
Let 
$$ J   = \| x f_{xx} \|^2_{L^2[0,\infty]} + c_1  \|x f_x\|^2_{L^2[0,\infty]} + c_2\|x f\|^2_{L^2[0,\infty]} + c_3 \|f\|^2_{L^2[0,\infty]}.$$
Using the argument from the proof of Lemma 1 in conjunction with
$$ \|f\|_{L^2} = 1 = \sum_{n=1}^{\infty}{|\left\langle f, v_n \right\rangle|^2},$$
we can conclude the existence of a constant $ 0 < c_5 < \infty$ depending only on $c_4$ such that
$$   \sum_{n \leq c_5 \sqrt{J}}^{}{ |\left\langle f, v_n \right\rangle|^2} \geq \frac{\|f\|^2_{L^2}}{2}.$$
The argument now follows from the exponential decay of the singular values (see \cite{led0}) and the elementary inequality $(a^2+b^2+c^2+d^2)^{1/2} \leq a+b+c+d$ for positive $a,b,c,d \in \mathbb{R}_{\geq 0}$
 $$ \sqrt{J} \leq  \| x f_{xx} \|_{L^2[0,\infty]} + c_1  \|x f_x\|_{L^2[0,\infty]} + c_2\|x f\|_{L^2[0,\infty]} + c_3 \|f\|_{L^2[0,\infty]}.$$
\end{proof}

\section{Proof of Theorem 3}

\subsection{The Differential Operator} 
Consider the self-adjoint operator
 $\mathcal{F}_T: L^2[-1,1] \rightarrow L^2[-1,1]$
$$ (\mathcal{F}_Tf)(x) =  \int_{-1}^{1}{f(x) e^{i  \xi x}dx}.$$
The crucial ingredient, which the monograph of Osipov, Rokhlin \& Xiao \cite{mono} ascribes to Landau \& Pollak \cite{pr2, pr3}  and Slepian \& Pollak \cite{pr1}, is that
the eigenfunctions of $\mathcal{F}_T$ coincide with the eigenfunctions of a differential operator.
\begin{lemma}[\cite{pr2, pr3, pr1}]  The eigenfunctions $(u_n)_{n=1}^{\infty}$ of $\mathcal{F}_T$ coincide with the eigenfunctions of 
$$D = -(1-x^2)\frac{d^2}{dx^2} + 2x\frac{d}{dx} +  x^2 \qquad \mbox{on} ~ [-1,1].$$
\end{lemma}
It is classical that the eigenvalues of the differential operator grow asymptotically as $\lambda_n \sim_{} n^2$, in particular, we have $\lambda_n \geq c_3 n^2$ for some $c_3 > 0$.

\subsection{Proof of Theorem 3} 
\begin{proof} Let $f \in H^1[-1,1]$ be arbitrary. We have
$$
 \sum_{n=1}^{\infty}{c_3  n^2 |\left\langle f, u_n \right\rangle|^2} \leq \sum_{n=1}^{\infty}{\lambda_n |\left\langle f, u_n \right\rangle|^2} = \left\langle Df, f\right\rangle.
$$
Repeated integration by parts gives that
\begin{align*}  \left\langle Df, f\right\rangle &= \int_{-1}^{1}{(1-x^2)f_x(x)^2 + x  \frac{d}{dx}(f(x)^2) + x^2 f(x)^2 dx}\\
&= f(1)^2- f(-1)^2 + \int_{-1}^{1}{(1-x^2)f_x(x)^2 + (x^2-1) f(x)^2 dx} \\
&\leq \left[ f(1)^2 - f(-1)^2 \right] + \int_{-1}^{1}{f_x(x)^2 + f(x)^2 dx}.
  \end{align*}
We again distinguish cases: either $f$ changes sign or it does not. If $f$ has a root somewhere, then with Lemma 2 we may conclude that
$$ \max \left( f(1)^2, f(-1)^2, \int_{-1}^{1}{f(x)^2 dx} \right) \leq 4 \int_{-1}^{1}{f_x(x)^2 dx}$$
and the result follows as before. The difference in the final result is a result of the different asymptotical behaviour of the eigenvalues (see e.g. Widom \cite{widom}, a very
precise description of the asymptotic behavior can be found in Fuchs \cite{fuchs})
$$ \log{ \lambda_n} \sim - n \log{n}.$$
If $f$ does not change sign, we have to argue differently. Assume w.l.o.g. that $f \geq 0$. Then
\begin{align*}
\| \mathcal{F}_{T} f\|_{L^2[-1,1]}^2 &= \left\langle \mathcal{F}_T f, \mathcal{F}_T f\right\rangle_{L^2[-1,1]}  =
 \int_{-1}^{1}{  \left( \int_{-1}^{1}{ f(x) e^{i x \xi} dx} \right) \overline{  \left( \int_{-1}^{1}{ f(x) e^{i x \xi} dx} \right) } d\xi}\\
&=  \int_{-1}^{1}{   \int_{-1}^{1}{   \int_{-1}^{1}{  f(x) f(y) e^{i \xi (x-y)} d\xi} d x} dy} \\
&= \int_{-1}^{1}{   \int_{-1}^{1}{ \frac{2 \sin{(x-y)}}{x-y} f(x) f(y) dx dy}} \\
&\geq \frac{1}{2} \int_{-1}^{1}{   \int_{-1}^{1}{  f(x) f(y) dx dy}} = \frac{1}{2} \left(\int_{-1}^{1}{f(x) dx}\right)^2.
\end{align*}
It remains to show that for every $c_2 > 0$ there exists $c_1 > 0$ such that for all differentiable $f:[-1,1] \rightarrow \mathbb{R}$
that do not change sign
$$ \frac{1}{2} \left(\int_{-1}^{1}{f(x) dx}\right)^2 \geq c_1 \exp{\left(-c_2\frac{ \|f_x\|_{L^2[a,b]}}{\|f\|_{L^2[a,b]}}\right)}\|f\|^2_{L^2[a,b]}$$
which follows from Lemma 3.
\end{proof}

%\textbf{Acknowledgement.} 

\end{document}